\documentclass[11pt]{article}
\usepackage[T1]{fontenc}
\usepackage{tocbibind}
\usepackage{amsmath,amsthm,amssymb}
\usepackage{fancyhdr}
\usepackage[british]{babel}
\usepackage{geometry,mathtools}
\usepackage{enumitem}
\usepackage{algpseudocode}
\usepackage{dsfont}
\usepackage{centernot}
\usepackage{xstring}
\usepackage{colortbl}
\usepackage{titlefoot}
\usepackage[section]{algorithm}
\usepackage{graphicx}
\usepackage[font={footnotesize}]{caption}
\usepackage[usenames,dvipsnames,table]{xcolor}
\usepackage{pstricks,tikz}
\usepackage[h]{esvect}
\usepackage[
bookmarksopen=true,
bookmarksopenlevel=1,
colorlinks=true,
linkcolor=darkblue,
linktoc=page,
citecolor=darkblue,
]{hyperref}
\usepackage{bbm}

\numberwithin{equation}{section}

\definecolor{darkblue}{rgb}{0,0,0.5}

\newdimen\margin
\def\textno#1&#2\par{
	\margin=\hsize
	\advance\margin by -4\parindent
	\setbox1=\hbox{\sl#1}
	\ifdim\wd1 < \margin
	$$\box1\eqno#2$$
	\else
	\bigbreak
	\hbox to \hsize{\indent$\vcenter{\advance\hsize by -3\parindent
			\it\noindent#1}\hfil#2$}
	\bigbreak
	\fi}

\newtheorem{theorem}[algorithm]{Theorem}
\newtheorem{prop}[algorithm]{Proposition}
\newtheorem{lemma}[algorithm]{Lemma}

\theoremstyle{definition}

\newtheorem{remark}[algorithm]{Remark}

\def\lateproof#1{\removelastskip\penalty55\medskip\noindent\begin{stepenv}\end{stepenv}{\bf Proof of #1. }} 

\def\noproof{{\unskip\nobreak\hfill\penalty50\hskip2em\hbox{}\nobreak\hfill%
		$\square$\parfillskip=0pt\finalhyphendemerits=0\par}\goodbreak}
\def\endproof{\noproof\bigskip}

\newcounter{stepenv}
\newenvironment{stepenv}[1][]{\refstepcounter{stepenv}}{}

\newcounter{step}[stepenv]

\newcounter{substep}[step]
\renewcommand{\thesubstep}{\thestep.\arabic{substep}}

\newcounter{claim}[stepenv]

\newcommand{\cE}{\mathcal{E}}
\newcommand{\cF}{\mathcal{F}}

\newcommand{\bN}{\mathbb{N}}

\newcommand{\bR}{\mathbb{R}}

\def\eps{{\epsilon}}
\newcommand{\eul}{{e}}

\newcommand{\defn}{\emph}

\newcommand{\prob}[1]{\mathrm{\mathbb{P}}\left[#1\right]}
\newcommand{\cprob}[2]{\prob{#1 \;\middle|\; #2}}
\newcommand{\expn}[1]{\mathrm{\mathbb{E}}\left[#1\right]}

\def\sm{\setminus}
\newcommand{\Set}[1]{\{#1\}}
\newcommand{\set}[2]{\{#1\,:\;#2\}}
\def\In{\subset}

\newcommand{\IND}{\mathbbm{1}}

\def\COMMENT#1{}
\def\TASK#1{}
\let\TASK=\footnote             

\begin{document}

	\title{The largest hole in sparse random graphs}

	\author{
		Nemanja Dragani\'c \thanks{Department of Mathematics, ETH, 8092 Z\"urich, Switzerland.
			Email: \href{mailto:nemanja.draganic@math.ethz.ch}{\nolinkurl{nemanja.draganic@math.ethz.ch}}. Research supported in part by SNSF grant 200021\_196965.
		}		
		\and
		Stefan Glock \thanks{Institute for Theoretical Studies, ETH, 8092 Z\"urich, Switzerland.
			Email: \href{mailto:dr.stefan.glock@gmail.com}{\nolinkurl{dr.stefan.glock@gmail.com}}. Research supported by Dr.~Max R\"ossler, the Walter Haefner Foundation and the ETH Z\"urich Foundation.}
		\and
		Michael Krivelevich \thanks{School of Mathematical Sciences, Raymond and Beverly Sackler Faculty of Exact Sciences, Tel Aviv University, Tel Aviv, 6997801, Israel. Email: \href{mailto:krivelev@tauex.tau.ac.il}{\nolinkurl{krivelev@tauex.tau.ac.il}}. Research supported in part by USA-Israel BSF grant 2018267, and by ISF grant 1261/17.}
	}
	
	\unmarkedfntext{2020 \emph{Mathematics Subject Classification}. Primary 05C80, 05C38; Secondary 05C05, 05D40.}
	\unmarkedfntext{\emph{Key words and phrases}. random graph, induced path, hole.}
	
	\date{}
	
	\maketitle

	\begin{abstract} 
		We show that for any $d=d(n)$ with $d_0(\eps) \le d =o(n)$, with high probability, the size of a largest induced cycle in the random graph $G(n,d/n)$ is $(2\pm \eps)\frac{n}{d}\log d$. This settles a long-standing open problem in random graph theory.
	\end{abstract}

	\section{Introduction}
	
	Let $G(n,p)$ denote the binomial random graph on $n$ vertices, where each edge is included independently with probability~$p$. In this paper, we are concerned with \emph{induced} subgraphs of $G(n,p)$, specifically trees, forests, paths and cycles.
	
	The study of induced trees in $G(n,p)$ was initiated by Erd\H{o}s and Palka~\cite{EP:83} in the 80s. Among other things, they showed that for constant $p$, with high probability (\textbf{whp}) the size of a largest induced tree in $G(n,p)$ is asymptotically equal to $2\log_q(np)$, where $q=\frac{1}{1-p}$. The obtained value coincides asymptotically with the \emph{independence number} of $G(n,p)$, the study of which dates back even further to the work of Bollob\'as and Erd\H{o}s~\cite{BE:76}, Grimmett and McDiarmid~\cite{GM:75} and Matula~\cite{matula:76}.
	
	As a natural continuation of their work, Erd\H{o}s and Palka~\cite{EP:83} posed the problem of determining the size of a largest induced tree in \emph{sparse} random graphs, when $p=d/n$ for some fixed constant~$d$. More precisely, they conjectured that for every $d>1$ there exists $c(d)>0$ such that \textbf{whp} $G(n,p)$ contains an induced tree of order at least $c(d)\cdot n$.
	This problem was settled independently in the late 80s by Fernandez de la Vega~\cite{fernandez-de-la-Vega:86}, Frieze and Jackson~\cite{FJ:87a}, Ku\v{c}era and R\"{o}dl~\cite{KR:87} as well as \L{}uczak and Palka~\cite{LP:88}.
	In particular, Fernandez de la Vega~\cite{fernandez-de-la-Vega:86} showed that one can take $c(d)\sim \frac{\log d}{d}$, and a simple first moment calculation reveals that this is tight within a factor of~$2$. Here and throughout, $\log$ denotes the natural logarithm if no base is specified.
	
	Two natural questions arise from there. First, one might wonder whether it is possible to find not only some  \emph{arbitrary} induced tree, but a \emph{specific} one, say a long induced path. Indeed, Frieze and Jackson~\cite{FJ:87b} in a separate paper showed that \textbf{whp} there is an induced path of length $\tilde{c}(d)\cdot n$. Two weaknesses of this result were that their proof only worked for sufficiently large~$d$, and that the value obtained for $\tilde{c}(d)$ was far away from the optimal one.
	Later, \L{}uczak~\cite{luczak:93} and Suen~\cite{suen:92} independently remedied this situation twofold. They proved that an induced path of length linear in $n$ exists for all $d>1$, showing that the conjecture of Erd\H{o}s and Palka holds even for induced paths. Moreover, they showed that one can take $\tilde{c}(d)\sim \frac{\log d}{d}$ as in the case of arbitrary trees. 
	
	A second obvious question is to determine the size of a largest induced tree (and path) more precisely. The aforementioned results were proved by analysing the behaviour of certain constructive algorithms which produce large induced trees and paths. The value $\frac{\log d}{d}$ seems to constitute a natural barrier for such approaches. On the other hand, recall that in the dense case, the size of a largest induced tree coincides asymptotically with the independence number. In 1990, Frieze~\cite{frieze:90} showed that the first moment bound $\sim2\frac{n}{d}\log d$ is tight for the independence number, also in the sparse case. His proof is based on the profound observation that the second moment method can be used even in situations where it apparently does not work, if one can combine it with a strong concentration inequality.
	Finally, in 1996, Fernandez de la Vega~\cite{fernandez-de-la-Vega:96} observed that the earlier achievements around induced trees can be combined with Frieze's breakthrough to prove that the size of a largest induced tree is indeed $\sim 2\frac{n}{d}\log d$. This complements the result of Erd\H{o}s and Palka~\cite{EP:83} in the dense case. (When $p=o_n(1)$, we have $2\log_q(np)\sim 2\frac{n}{d}\log d$.)
	
	Fernandez de la Vega~\cite{fernandez-de-la-Vega:96} also posed the natural problem of improving the \L{}uczak--Suen bound~\cite{luczak:93,suen:92} for induced paths, for which his approach was ``apparently helpless''. Despite the widely held belief (see~\cite{CDKS:ta,DS:18} for instance) that the upper bound $\sim 2\frac{n}{d}\log d$ obtained via the first moment method is tight, the implicit constant $1$ has not been improved in the last 30 years.

	\subsection{Long induced paths and cycles}\label{sec:paths}
	
	Our main result is the following, which settles this problem and gives an asymptotically optimal result for the size of a largest induced path in $G(n,p)$. 
	
	\begin{theorem}\label{thm:path}
		For any $\eps>0$ there is $d_0$ such that \textbf{whp} $G(n,p)$ contains an induced path of length at least $(2-\eps)\frac{n}{d}\log d$ whenever $d_0\le d=pn =o(n)$.
	\end{theorem}
	
	For the sake of generality, we state our result for a wide range of functions~$d=d(n)$. 
	However, we remark that the most interesting case is when $d$ is a sufficiently large constant. 
	In fact, for dense graphs, when $d\ge n^{1/2}\log^2 n$, more precise results are already known (cf.~\cite{DS:18,rucinski:87}).
	
	Some of the earlier results~\cite{DS:18,FJ:87b,luczak:93} are phrased in terms of induced cycles (\defn{holes}). This does not make the problem much harder (see Remark~\ref{rem:cycle}).
	We also note that our proof is self-contained, except for well-known facts from probability and graph theory.
	
	We now briefly outline our strategy.
	Roughly speaking, the idea is to find a long induced path in two steps. First, we find many disjoint paths of some chosen length $L=L(d)$, such that the subgraph consisting of their union is induced. To achieve this, we generalize a recent result of Cooley, Dragani\'c, Kang and Sudakov~\cite{CDKS:ta} who obtained large induced matchings. We will discuss this further in Section~\ref{sec:forests}.
	Assuming now we can find such an induced linear forest $F$, the aim is to connect almost all of the small paths into one long induced path, using a few additional vertices.\footnote{Recall that a forest is called \defn{linear} if its components are paths.} As a ``reservoir'' for these connecting vertices, we find (actually, even before finding $F$) a large independent set $I$ which is disjoint from~$F$.
	To model the connecting step, we give each path in $F$ a direction, and define an auxiliary digraph whose vertices are the paths, and two paths $(P_1,P_2)$ form an edge if there exists a ``connecting'' vertex $a\in I$ that has some edge to one of the last $\eps L$ vertices of $P_1$ and some edge to one of the first $\eps L$ vertices of $P_2$, but no edge to the rest of~$F$.
	Our goal is to find an almost spanning path in this auxiliary digraph. Observe that this will provide us with a path in $G(n,p)$ of length roughly~$|F|$.
	The intuition is that the auxiliary digraph behaves quite randomly, which gives us hope that, even though it is very sparse, we can find an almost spanning path. 
	
	Crucially, we do not perform this connecting step in the whole random graph. This is because ensuring that the new connecting vertices are only connected to two vertices of $F$ is too costly, making the auxiliary digraph so sparse that it is impossible to find an almost spanning path. Instead, we use a sprinkling argument, meaning that we view $G(n,p)$ as the union of two independent random graphs $G_1$ and $G_2$, where the edge probability of $G_2$ is much smaller than~$p$. We then reveal the random choices in several stages. When finding $F$ and $I$ as above, we make sure that there are no $G_1$-edges between $F$ and~$I$. Then, in the final connecting step, it remains to expose the $G_2$-edges between $F$ and $I$, with the advantage that now the edge probability is much smaller, making it much more suitable for a desired ``sparse'' connection.
	
	\subsection{Induced forests with small components}\label{sec:forests}
	
	As outlined above, in the first step of our argument, we seek an induced linear forest whose components are reasonably long paths.
	For this, we generalize a recent result of Cooley, Dragani\'c, Kang and Sudakov~\cite{CDKS:ta}. They proved that \textbf{whp} $G(n,p)$ contains an induced matching with $\sim 2\log_q (np)$ vertices, which is asymptotically best possible. They also anticipated that using a similar approach one can probably obtain induced forests with larger, but bounded components. As a by-product, we confirm this.
	To state our result, we need the following definition. For a given graph $T$, a \defn{$T$-matching} is a graph whose components are all isomorphic to~$T$. Hence, a $K_2$-matching is simply a matching, and specifying the following statement for $T=K_2$ implies the main result of~\cite{CDKS:ta}. 
	
	\begin{theorem}\label{thm:forests}
		For any $\eps>0$ and any fixed tree $T$, there exists $d_0>0$ such that \textbf{whp} the order of the largest induced $T$-matching in $G(n,p)$ is $(2\pm \eps)\log_q(np)$, where $q=\frac{1}{1-p}$, whenever $\frac{d_0}{n}\le p\le 0.99$.
	\end{theorem}

	We use the same approach as in~\cite{CDKS:ta}, which goes back to the work of Frieze~\cite{frieze:90} (see also~\cite{bollobas:88,SS:87}).
	The basic idea is as follows. Suppose we have a random variable $X$ and want to show that \textbf{whp}, $X\ge b-t$, where $b$ is some ``target'' value and $t$ a small error. 
	For many natural variables, we know that $X$ is ``concentrated'', say $\prob{|X- \expn{X}| \ge  t/2} < \rho$ for some small~$\rho$. This is the case for instance when $X$ is determined by many independent random choices, each of which has a small effect.
	However, it might be difficult to estimate $\expn{X}$ well enough.
	But if we know in addition that $\prob{X\ge b}\ge \rho$, then we can combine both estimates to $\prob{X\ge b}> \prob{X\ge \expn{X}+t/2}$, which clearly implies that $b\le \expn{X}+t/2$. Now applying the other side of the concentration inequality, we infer $\prob{X\le b-t} \le \prob{X\le \expn{X}-t/2}< \rho$, as desired.
	
	In our case, say $X$ is the maximum order of an induced $T$-matching in~$G(n,p)$. Since adding or deleting edges at any one vertex can create or destroy at most one component, we know that $X$ is $|T|$-Lipschitz and hence concentrated (see Section~\ref{sec:concentration}). Using the above approach, it remains to complement this with a lower bound on the probability that $X\ge b$. Introduce a new random variable $Y$ which is the \defn{number} of induced $T$-matchings of order~$b$ (a multiple of $|T|$). Then we have $X\ge b$ if and only if $Y>0$. The main technical work is to obtain a lower bound for the probability of the latter event using the second moment method.
	We note that by applying the second moment method to labelled copies (instead of unlabelled copies as in~\cite{CDKS:ta}) we obtain a shorter proof even in the case of matchings (see Section~\ref{sec:2ndmoment}).
	More crucially, it turns out that one can even find induced forests where the component sizes can grow as a function of~$d$, which we need in the proof of Theorem~\ref{thm:path} (specifically, we need $L\gg \log d$). This is provided by the following auxiliary result. We note that the same holds for forests with arbitrary components of bounded degree, and one can also let the degree slowly grow with~$d$, but we choose to keep the presentation simple.
	
	\begin{lemma}\label{lem:forests}
		For any $\eps>0$, there exists $d_0>0$ such that \textbf{whp} $G(n,p)$ contains an induced linear forest of order at least $(2-\eps)p^{-1}\log(np)$ and component paths of order $d^{1/2}/\log^4 d$, whenever $d_0\le d=np \le n^{1/2}\log^2 n$. 
	\end{lemma}

	\subsection{Notation}
	
	We use standard graph theoretical notation. In particular, for a graph $G$ and $U\In V(G)$, we let $|G|$ denote the order of $G$, $e(G)$ the number of edges in $G$, $\Delta(G)$ the maximum degree and $G[U]$ the subgraph induced by~$U$.
	
	For functions $f(n),g(n)$, we write $f\sim g$ if $\lim_{n\to\infty}\frac{f(n)}{g(n)}=1$. We also use the standard Landau symbols $o(\cdot),\Omega(\cdot),\Theta(\cdot),O(\cdot),\omega(\cdot)$, where subscripts disclose the variable that tends to infinity if this is not clear from the context.
	We use $\approx$ non-rigorously in informal discussions and ask the reader to interpret it correctly.
	
	An event $\cE_n$ holds \defn{with high probability} (\textbf{whp}) if $\prob{\cE_n}= 1-o_n(1)$.
	Also, $[n]=\Set{1,\dots,n}$ and $(n)_k=n(n-1)\cdots (n-k+1)$.
	As customary, we tacitly treat large numbers like integers whenever this has no effect on the argument.

	\section{Second moment}\label{sec:2ndmoment}
	
	In this section, we use the second moment method to derive a lower bound on the probability that $G(n,d/n)$ contains a given induced linear forest of size $\sim 2\frac{n}{d}\log d$. Here, it does not matter whether the components are small.
	More precisely, we prove that for fixed $\eps>0$ and $d\ge d_0(\eps)$, \emph{any} bounded degree forest $F$ on $k\le (2-\eps)\frac{n}{d}\log d$ vertices is an induced subgraph of $G(n,d/n)$ with probability at least $\exp(-O(\frac{n\log^2 d}{d^2}))$.
	Moreover, when $d=\omega(n^{1/2}\log n)$, the obtained probability bound tends to~$1$. In particular, in this regime, the lemma readily implies the existence of an induced path of the asymptotically optimal length $\sim 2\frac{n}{d}\log d$ \textbf{whp}.

	\begin{lemma}\label{lem:2ndmoment}
		For any $\eps>0$, there exists $d_0$ such that the following holds for all $d_0\le d < n$, where $p=\frac{d}{n}$ and $q=\frac{1}{1-p}$. For any forest $F$ on $k\le (2-\eps)\log_q d$ vertices with maximum degree $\Delta\le d^{\eps/6}$, the probability that $G(n,p)$ contains an induced copy of $F$ is at least $$\exp\left(-10^4\Delta^2\frac{n\log^2 d}{d^2} -2d^{-\eps/7}\right) .$$
	\end{lemma}

	The proof of Lemma~\ref{lem:2ndmoment} is based on the second moment method and will be given below. We start off with some basic preparations which will also motivate the main counting tool.
	
	Fix a forest~$F$ of order~$k$. Let $Y$ be the random variable which counts the number of \emph{labelled} induced copies of $F$ in $G(n,p)$. More formally, let $\cF$ be the set of all injections $\sigma\colon V(F)\to [n]$, and for $\sigma\in \cF$, let $F_\sigma$ be the graph with vertex set $\set{\sigma(x)}{x\in V(F)}$ and edge set $\set{\sigma(x)\sigma(y)}{xy\in E(F)}$.
	Let $A_\sigma$ be the event that $F_\sigma$ is an induced subgraph of~$G(n,p)$.
	Hence, $$\prob{A_\sigma}=p^{e(F)}(1-p)^{\binom{k}{2}-e(F)},$$
	and setting $Y=\sum_{\sigma\in \cF}\IND(A_\sigma)$, we have 
	\begin{align}
		\expn{Y}=(n)_k p^{e(F)}(1-p)^{\binom{k}{2}-e(F)}.\label{expn}
	\end{align}
	Ultimately, we want to obtain a lower bound for $\prob{Y>0}$. 
	Fix some $\sigma_0\in \cF$. By symmetry, the second moment of $Y$ can be written as $$\expn{Y^2}=\expn{Y} \sum_{\sigma\in \cF}\cprob{A_\sigma}{A_{\sigma_0}}.$$
	Applying the Paley--Zygmund inequality, we thus have 
	\begin{align}
		\prob{Y>0}\ge \frac{\expn{Y}^2}{\expn{Y^2}} = \frac{\expn{Y}}{\sum_{\sigma\in \cF}\cprob{A_\sigma}{A_{\sigma_0}}}.\label{2ndmoment inequality}
	\end{align}

	The remaining difficulty is to control the terms $\cprob{A_\sigma}{A_{\sigma_0}}$.
	We say that $\sigma\in \cF$ is \defn{compatible} (with $\sigma_0$) if $\cprob{A_\sigma}{A_{\sigma_0}}>0$. This means that, in the intersection $V(F_\sigma)\cap V(F_{\sigma_0})$, a pair $uv$ which is an edge in $F_\sigma$ cannot be a non-edge in $F_{\sigma_0}$, and vice versa, as otherwise $F_\sigma$ and $F_{\sigma_0}$ could not be induced subgraphs of $G(n,p)$ simultaneously. From now on, we can ignore all 
	$\sigma$ that are not compatible with $\sigma_0$.
	
	If $\sigma\in\cF$ is compatible with $\sigma_0$, we denote by $I_\sigma:=F_\sigma \cap F_{\sigma_0}$ the graph on $S=V(F_\sigma)\cap V(F_{\sigma_0})$ with edge set $E(F_\sigma[S])=E(F_{\sigma_0}[S])$. This ``intersection graph'' assumes a crucial role in the analysis.
	Suppose that $I_\sigma$ has $s$ vertices and $c$ components. Since $I_\sigma$ is a forest, we have $e(I_\sigma)=s-c$. These are the edges of $F_\sigma$ that we already know to be there when conditioning on~$A_{\sigma_0}$, and for $F_\sigma$, we need $e(F)-e(I_\sigma)$ ``new'' edges. 
	Moreover, there are $\binom{k}{2}-\binom{s}{2}-e(F)+e(I_\sigma)$ additional non-edges. 
	Therefore,
	\begin{align}
		\cprob{A_\sigma}{A_{\sigma_0}} = p^{e(F)-s+c} (1-p)^{\binom{k}{2}-\binom{s}{2}-e(F)+s-c}.\label{cond prob}
	\end{align}
	Note here that when the number of components $c$ is large, then the exponent of $p$ is large and hence we have a stronger upper bound on $\cprob{A_\sigma}{A_{\sigma_0}}$. On the other hand, if $c$ is small, then $\cprob{A_\sigma}{A_{\sigma_0}}$ is larger, but this will be compensated by the fact that there are fewer such~$\sigma$.
	In the following, we bound the number of compatible $\sigma\in\cF$ for which $I_\sigma$ has $s$ vertices and $c$ components. We remark that this kind of analysis was also carried out in~\cite{draganic:20} in the study of dense random graphs. We include the details for completeness, with an improved dependence on~$\Delta$.
	We make use of the following well-known counting result, see, e.g., \cite[Lemma~2]{BFMD:98} for the proof. 
	
	\begin{prop}\label{prop:branching}
		For a graph $H$ with $\Delta(H)\le \Delta$ and $v\in V(H)$, the number of (unlabelled) trees in $H$ of order $s$ which contain $v$ is at most $(e\Delta)^{s-1}$.
	\end{prop}

	\begin{prop}\label{prop:counting extension}
		For all $0\le c\le s$, the number of compatible $\sigma\in\cF$ for which $I_\sigma$ has $s$ vertices and $c$ components is at most $$\binom{k}{c} k^c (6\Delta^2)^s (n-k)_{k-s}.$$
	\end{prop}
	
	\begin{proof}
		Fix $s$ and~$c$. We can obviously assume that $s\ge c\ge 1$, as the only other non-void case is when $s=c=0$, in which case the expression $(n-k)_k$ is exactly correct. 
		The first claim is that the number of subgraphs of $F_{\sigma_0}$ with $s$ vertices and $c$ components is at most $\binom{k}{c} (2e\Delta)^s$. 
		To see this, we first choose root vertices $v_1,\dots,v_c$ for the components, for which there are at most $\binom{k}{c}$ choices. For $i\in [c]$, let $T_i$ denote the component of $I_\sigma$ which will contain~$v_i$. Next, we fix the sizes of the components. Writing $s_i=|T_i|$, the number of possibilities is bounded by the number of positive integer solutions of $s_1+\dots+s_c=s$, which is $\binom{s-1}{c-1}\le 2^s$ by a well-known formula. 
		Now, having fixed the sizes, we can apply Proposition~\ref{prop:branching} for each $i\in[c]$, with $F_{\sigma_0},v_i$ playing the roles of $H,v$, to see that the number of choices for $T_i$ is at most $(e\Delta)^{s_i-1}$, which combined amounts to $(e\Delta)^{s-c}$.
		This implies the claim, and immediately yields an upper bound on the number of possibilities for the intersection graph~$I_\sigma$.
		
		Now, fix a choice of $I_\sigma$. Since $I_\sigma$ is a forest with $c$ components, its vertices can be ordered such that every vertex, except for the first $c$ vertices in the ordering (which we can set to be the root vertices of the components), has exactly one neighbour preceding it.
		In order to count the number of possibilities for $\sigma$, we proceed as follows. 
		First, choose the preimages under $\sigma$ for the first $c$ vertices, for which there are at most $(k)_c$ choices. Now, we choose the preimages of the remaining vertices of $I_\sigma$ one-by-one in increasing order. In each step, there are at most $\Delta$ choices, since one neighbour of the current vertex has already chosen its preimage, and $I_\sigma$ has to be an induced subgraph of $F_\sigma$. Hence, there are at most $\Delta^{s-c}$ choices for the preimages of the remaining vertices of~$I_\sigma$.
		Finally, we have used $s$ vertices of $F$ as preimages for the vertices in $I_\sigma$. The remaining $k-s$ vertices of $F$ must be mapped to $[n]\sm V(F_{\sigma_0})$, so there are at most $(n-k)_{k-s}$ possibilities.
	\end{proof}

	With the preparations done, the proof of the lemma reduces to a chain of estimates.
	
	\lateproof{Lemma~\ref{lem:2ndmoment}}
	By~\eqref{2ndmoment inequality}, it suffices to show that
	\begin{align*}
		\frac{\sum_{\sigma\in \cF}\cprob{A_\sigma}{A_{\sigma_0}}}{\expn{Y}} \le \exp\left(10^4\Delta^2\frac{n\log^2 d}{d^2} + 2d^{-\eps/7}\right).
	\end{align*}
	
	We split the sum over compatible $\sigma\in \cF$ according to the number of vertices and components of~$I_\sigma$. Applying \eqref{expn},~\eqref{cond prob} and Proposition~\ref{prop:counting extension}, we obtain
	\begin{align*}
		\frac{\sum_{\sigma\in \cF}\cprob{A_\sigma}{A_{\sigma_0}}}{\expn{Y}} &\le
		\sum_{s=0}^k \sum_{c=0}^s \frac{\binom{k}{c}k^c (6\Delta^2)^s (n-k)_{k-s}  p^{e(F)-s+c} (1-p)^{\binom{k}{2}-\binom{s}{2}-e(F)}}{(n)_k p^{e(F)}(1-p)^{\binom{k}{2}-e(F)}} \\
		&= \sum_{s=0}^k \frac{(n-k)_{k-s}}{(n)_k}p^{-s}q^{\binom{s}{2}}(6\Delta^2)^s  \sum_{c=0}^s \binom{k}{c} (kp)^c.
	\end{align*}
	In order to bound this expression, we make some observations. Note that we always have 
	\begin{align}
		k\le 2\log_q(np)\le 2\frac{n}{d}\log d\label{upper bound nicer}
	\end{align}
	since $\log q\ge p$. Hence, $kp\le 2\log d$. Moreover, we have
	$\frac{(n-k)_{k-s}}{(n)_k} \le \frac{1}{(n)_s} \le (4/n)^s$ since $s\le k \le 3n/4$, say,  using~\eqref{upper bound nicer}. Observe further that $\binom{k}{c} \le \binom{2k}{c}\le \binom{2k}{s}\le \frac{(2k)^s}{s!}$.
	Finally, $c$ takes only $s+1\le 2^s$ values.
	Hence, the last sum displayed above is at most 
	\begin{align*}	
		\sum_{s=0}^k (4/n)^s p^{-s} q^{s^2/2} (6\Delta^2)^s \frac{(16k\log d)^s}{s!}
		&= \sum_{s=0}^k \frac{\left(\frac{384\Delta^2 k \log d}{d} q^{s/2}  \right)^s}{s!}.
	\end{align*}
	
	
	We split the final sum into two terms.
	First, consider the range $s\le k/\log d$. Then $q^{s/2}\le q^{1/\log q}= \eul$. Hence, recalling that $e^x=\sum_{s\ge 0}\frac{x^s}{s!}$, we obtain the bound
	$$\sum_{s=0}^{\lfloor k/\log d\rfloor } \frac{\left(\frac{384\Delta^2 k \log d}{d} q^{s/2}  \right)^s}{s!}\le \exp\left(\frac{384\eul \Delta^2 k\log d}{d} \right) \le \exp\left(\frac{10^4 \Delta^2 n\log^2 d}{d^2} \right).$$ 
	
	Finally, for $s\ge k/\log d$, we use $s!\ge (s/\eul)^s$ to bound each summand as
	\begin{align*}
		\frac{\left(\frac{384\Delta^2 k \log d}{d} q^{s/2}  \right)^s}{s!} \le \left(\frac{384\eul \Delta^2 k\log d}{ds} q^{s/2} \right)^s \le \left(\frac{384\eul \Delta^2 \log^2 d}{d}  q^{s/2} \right)^s.
	\end{align*}
	Crucially, since $s\le k\le (2-\eps)\log_q d$, we have $q^{s/2} \le q^{(1-\eps/2)\log_q d} = d^{1-\eps/2}$. Now, for sufficiently large $d\ge d_0$ the exponent base above is bounded by $\frac{384\eul \Delta^2 \log^2 d}{d^{\eps/2}} \le d^{-\eps/7}<1$. Therefore the geometric series tells us that 
	$$\sum_{s=\lceil k/\log d \rceil}^{k} \frac{\left(\frac{384\Delta^2 k \log d}{d} q^{s/2}  \right)^s}{s!}  \le \frac{1}{1-d^{-\eps/7}} -1 \le 2d^{-\eps/7}.$$ 
	
	Altogether, we conclude that
	\begin{align*}
		\frac{\sum_{\sigma\in \cF}\cprob{A_\sigma}{A_{\sigma_0}}}{\expn{Y}} \le  \exp\left(\frac{10^4 \Delta^2 n\log^2 d}{d^2} \right) + 2d^{-\eps/7} \le \exp\left(\frac{10^4 \Delta^2 n\log^2 d}{d^2} + 2d^{-\eps/7}\right),
	\end{align*}
	completing the proof.
	\endproof

	\section{Concentration}\label{sec:concentration}
	
	In this section, we deduce Theorem~\ref{thm:forests} and Lemma~\ref{lem:forests} from Lemma~\ref{lem:2ndmoment}. 
	We will use Talagrand's inequality.\COMMENT{One could also make the argument work with Azuma's inequality, but Talagrand's inequality is more convenient to use.}
	To state it, we need the following definitions.
	Given a product probability space $\Omega=\prod_{i=1}^n \Omega_i$ (endowed with the product measure) and a random variable $X\colon \Omega\to \bR$, we say that $X$ is
	\begin{itemize}
		\item \defn{$L$-Lipschitz} (for some $L>0$) if for any $\omega,\omega'\in \Omega$ which differ only in one coordinate, we have $|X(\omega)-X(\omega')|\le L$;
		\item \defn{$f$-certifiable} (for a function $f\colon \bN\to \bN$) if for every $s$ and $\omega\in\Omega$ such that $X(\omega)\ge s$, there exists a set $I\In [n]$ of size $\le f(s)$ such that $X(\omega')\ge s$ for every $\omega'$ that agrees with $\omega$ on the coordinates indexed by~$I$.
	\end{itemize}
	
	\begin{theorem}[Talagrand's inequality, see~\cite{AS:16}]
		Suppose that $X$ is $L$-Lipschitz and $f$-certifiable. Then, for all $b,t\ge 0$,
		$$\prob{X\le b-tL\sqrt{f(b)}} \prob{X\ge b} \le \exp\left(-t^2/4\right).$$
	\end{theorem}

	Our probability space is of course $G(n,p)$. Although this comes naturally as a product of $\binom{n}{2}$ elementary probability spaces $\Omega_{ij}$, one for each potential edge $ij$, it can be more effective, depending on the problem, to consider a description that is vertex-oriented, where the edges incident to a vertex are combined into one probability space (``vertex exposure'').
	Concretely, for $i\in [n-1]$, let $\Omega_i=\prod_{j>i}\Omega_{ij}$ represent all edges from vertex $i$ to vertices $j>i$.
	Then $G(n,p)=\prod_{i=1}^{n-1}\Omega_i$. Note here that the vertices are ordered to describe the product space in a way that every edge appears exactly once. Apart from that, this ordering plays no role.

	\lateproof{Theorem~\ref{thm:forests}}
	Fix $\eps>0$, a tree~$T$, and assume $d_0$ is sufficiently large. Let $L=|V(T)|$ and $d=np$. We first show the upper bound, by arguing that \textbf{whp} every set of size at least $t=(2+\eps)\log_q (np)$ spans at least $t$ edges, which prevents us from finding any induced forest of order $t$.
	Indeed, the probability that a fixed $t$-set  spans at most $t$ edges is at most  \[(t+1)\binom{\binom{t}{2}}{t}p^t(1-p)^{\binom{t}{2}-t},\]
	as the number of edges in such a set follows a binomial distribution $\text{Bin}\left(\binom{t}{2},p\right)$, with the mean being larger than $t$ (where we used that $p\le 0.99$). Hence we have that the expected number of $t$-sets which span at most $t$ edges is at most 
	\begin{align*}
		(t+1)\binom{n}{t}\binom{\binom{t}{2}}{t}p^t(1-p)^{\binom{t}{2}-t}&\leq (t+1)\left(\frac{en}{t}\right)^t\left(\frac{e\binom{t}{2}}{t}\right)^t\left(\frac{d}{n}\right)^t d^{-(1+\eps/3)t}\\
		&<(t+1) (e^2/2)^t  d^{-\eps t/3}=o(1),
	\end{align*}
	where we used standard estimates and the fact that $(1-p)^{t}=q^{-t}=d^{-(2+\eps)}$.
	
	We now turn to the lower bound.
	Let $X$ be the maximum order of an induced $T$-matching in $G(n,p)$.
	Our goal is to show that $X\ge (2-\eps)\log_q d$ \textbf{whp}. Set $b=(2-\eps/2)\log_q d$. 
	
	First, by Lemma~\ref{lem:2ndmoment}, we have 
	\begin{align*}
		\prob{X\ge b} \ge \exp\left(-10^4L^2\frac{n\log^2 d}{d^2} -2d^{-\Omega(\eps)}\right).
	\end{align*}
	This means that in the case $d\ge n^{1/2}\log^{2}n$, we are already done. Assume now that $d\le n^{1/2}\log^{2}n$. Then the above bound simplifies to 
	\begin{align}
		\prob{X\ge b} \ge \exp\left(-\frac{n\log^3 d}{d^2}\right).\label{2nd moment}
	\end{align}
	Recall also that in the regime $d=o(n)$ we have $\log_q d\sim \frac{n}{d}\log d$.
	
	It is easy to check that $X$ is $L$-Lipschitz and $f$-certifiable, where $f(s)=s+L$. Indeed, adding or deleting edges arbitrarily at one vertex can change the value of $X$ by at most~$L$, hence $X$ is $L$-Lipschitz. Moreover, if $X\ge s$, this means there is a set $I\In[n]$ of size $s\le |I|< s+L$ which induces a $T$-matching. If we leave the coordinates indexed by $I$ unchanged, this means in particular that $I$ still induces a $T$-matching, hence we still have $X\ge s$.
	
	Hence, Talagrand's inequality applied with $t=\frac{\sqrt{n}\log^3 d}{d}$ yields
	$$\prob{X\le b-tL\sqrt{b+L}} \prob{X\ge b} \le \exp\left(-\frac{n\log^6 d}{4d^2}\right).$$
	Together with~\eqref{2nd moment} and since 
	\begin{align}
		tL\sqrt{b+L}\le \frac{\sqrt{n}\log^3 d}{d} L \sqrt{2\frac{n}{d}\log d}\le \frac{n}{d},\label{Lipschitz}
	\end{align}
	we infer that the probability of $X\le b-\frac{n}{d}$ is at most $\exp\left(-\frac{n\log^6 d}{5d^2}\right)=o_n(1)$.
	This completes the proof since $b-\frac{n}{d}\ge (2-\eps)\log_q d$.
	\endproof

	In the above proof, we had some room to spare in~\eqref{Lipschitz}. We will now exploit this to allow the component sizes to grow with~$d$. The proof is almost verbatim the same, so we only point out the differences.

	\lateproof{Lemma~\ref{lem:forests}}
	Note that we are only interested in the case $d\le n^{1/2}\log^{2}n$ and when $T$ is a path of order~$L$. Since $\Delta(T)\le 2$, Lemma~\ref{lem:2ndmoment} still provides the lower bound in~\eqref{2nd moment}.
	All we have to ensure now is that~\eqref{Lipschitz} still holds, and this is easily seen to be the case as long as $L\le d^{1/2}/\log^4 d$.
	\endproof

	\section{Connecting}\label{sec:connect}
	In this section, we use Lemma~\ref{lem:forests} to prove Theorem~\ref{thm:path} as outlined in Section~\ref{sec:paths}. Recall that we intend to define an auxiliary digraph on the components of a linear forest, where an edge corresponds to a suitable connection between two paths. 
	Our goal is to find an almost spanning path in this random digraph. The tool which enables us to achieve this, Lemma~\ref{lem:DFS} below, is based on the well-known graph exploration process \emph{depth-first-search} (DFS).
	The usefulness of DFS to find long paths in random graphs was demonstrated by Krivelevich and Sudakov~\cite{KS:13} in a paper where they give rather short and simple proofs of classical results in random graph theory. In our proof, we use the following straightforward consequence of DFS.
	
	\begin{lemma}[{\cite[Lemma~4.4]{BEKS:12}}]\label{lem:DFS}
		Let $D$ be a digraph on $N$ vertices and suppose that for any two disjoint sets $S,T\In V(D)$ of size $k$, there exists an edge directed from $S$ to~$T$. Then $D$ contains a path of length $N-2k+1$.
	\end{lemma}

	\lateproof{Theorem~\ref{thm:path}}
	Fix $\eps>0$ and assume that $d\ge d_0$ is sufficiently large.
	We will assume that $d=o(n^{1/2}\log^2 n)$. For the case $d=\omega(n^{1/2}\log n)$, Lemma~\ref{lem:2ndmoment} implies that \textbf{whp} there exists an induced path of length $(2-\eps)\frac{n}{d}\log d$.
	
	We expose the random graph in several stages, and will after each step fix an outcome that holds with high probability.
	First, we consider $G\sim G(n,p)$ as the union of two independent random graphs $G_1$ and $G_2$, where $G_2\sim G(n,p_2)$ with $p_2=\frac{d}{n\log d}$, and $G_1\sim G(n,p_1)$ with $p_1$ such that $1-p=(1-p_1)(1-p_2)$. (Hence, $G_1\cup G_2$ has the same distribution as $G$.)
	Note that clearly $p_1\le p=d/n$.
	
	Fix a subset $V_0\In [n]$ of size $\frac{n}{2d}$. Now, in the first exposure round, we reveal all random edges from $G=G_1\cup G_2$ inside~$V_0$.
	The expected number of edges is at most $\frac{n}{8d}$, and using a standard Chernoff bound, it is easily seen that \textbf{whp}, the number of edges inside $V_0$ is at most $\frac{n}{6d}$, say. From now on, fix such an outcome. By deleting a vertex from each edge, we can find an independent set $I\In V_0$ in $G[V_0]$ of size $\frac{n}{3d}$.
	
	In the second round, we expose all edges with one vertex in $I$ and the other in $[n]\setminus V_0$, but only those which belong to~$G_1$. The expected number of such edges is at most~$n/3$.
	Again using a Chernoff bound, the number of $G_1$-edges leaving $I$ is at most $2n/5$ \textbf{whp}. From now on, fix such an outcome.
	Let $V_1=[n]\sm (I\cup N_{G_1}(I))$. Since $|N_{G_1}(I)|+|I|\le n/2$, we have $|V_1|\ge n/2$.
	
	In the third step, we reveal the random edges of $G=G_1\cup G_2$ inside~$V_1$.
	Now we apply Lemma~\ref{lem:forests} to $G[V_1]$; we get that \textbf{whp} there exists an induced forest $F$ of order $$ \left(2-\eps \right)p^{-1}\log \left(np/2\right)\ge \left(2-2\eps\right)\frac{n}{d}\log d$$ whose components are paths of order $L=\Theta\left(\frac{\sqrt{d}}{\log^{4}d}\right)$. 
	Again, we fix such~$F$. Note that $F$ is induced in $G=G_1\cup G_2$ and that $I$ is independent in $G=G_1\cup G_2$. Moreover, by definition of $V_1$, we know that there are no edges in $G_1$ between $F$ and~$I$.
	
	In the fourth and final step, we reveal all the remaining random edges, which in particular includes the $G_2$-edges between $F$ and~$I$.
	Our goal is to use some vertices from $I$ to connect most of the paths of $F$ into one long induced path.
	To achieve this, we define the following auxiliary digraph~$D$.
	Give each of the component paths $P$ of $F$ an arbitrary direction, and let $P^-$ denote the initial $\eps L$ vertices and $P^+$ the last $\eps L$ vertices on~$P$.
	Now, the vertex set of $D$ is simply the set of components of~$F$. 
	For two paths $P_1,P_2$, we include $(P_1,P_2)$ as an edge in $D$ if there exists a vertex $a\in I$ such that $a$ has exactly one edge (of $G_2$) to both $P_1^+$ and $P_2^-$, but no other edge (of $G_2$) to any vertex of~$F$.
	Note that $D$ is a random digraph. Our claim is that, with high probability, it contains an almost spanning path.
	Let $N=|V(D)|=\Theta\left(\frac{n}{dL}\log d\right)$. Consider any two disjoint sets $S,T\In V(D)$ of size $\eps N$.
	For each $a\in I$ and all $P_1\in S,P_2\in T$, the probability that $a$ is a suitable connection from $P_1$ to $P_2$ is the sum of probabilities over the $(\eps L)^2$ possible pairs to form a suitable connection (since those events are disjoint) and equals to
	$$(\eps L p_2)^2 (1-p_2)^{|F|-2}\ge (\eps L p_2)^2 \eul^{-2p_2|F|} =  \Theta\left(\eps^2 L^2 p_2^2\right)$$
	since $p_2|F|=\Theta(1)$ by our choice of~$p_2$.
	Moreover, for distinct pairs $(P_1,P_2)$ and fixed $a$, these events are disjoint. Hence, the probability of $a$ giving some good connection from $S$ to $T$ is $\Theta(\eps^4 N^2 L^2 p_2^2)=\Theta(\eps^4)$.
	Finally, for distinct $a$, these events are determined by disjoint sets of edges, and hence independent, so the probability that there is no edge from $S$ to $T$ in $D$ is at most $$\left(1-\Theta(\eps^4)\right)^{|I|} \le \exp\left(-\Omega(\eps^4n/d) \right).$$
	The total number of choices for $S$ and $T$ is at most $4^N=\exp (\Theta(\frac{n}{dL}\log d))$. Thus, since we have that $L\gg \log d/\eps^4$, a union bound yields that \textbf{whp}, we can apply Lemma~\ref{lem:DFS} to get a path $P_1P_2\dots P_t$ in $D$ of length $(1-2\eps)N$. 
	This translates to an induced path of $G$ as follows: for each $i\in [t-1]$, since $P_iP_{i+1}\in E(D)$, there exists a vertex $a_i\in I$ which has exactly one edge (of $G_2$) to both $P_i^+$ and $P_{i+1}^-$, but no other edge (of $G_2$) to any vertex of~$F$. Clearly, the $a_i$'s are distinct, hence we obtain a path in $G$ in the obvious way (start with $P_1$, then from the appropriate vertex in $P_1^+$, go via $a_1$ to $P_2^-$, follow $P_2$, and then go from $P_2^+$ via $a_2$ to $P_3^-$, etc.). As remarked earlier, there are no $G_1$-edges between $F$ and $I$, hence by the definition of $E(D)$, the path will be induced. 
	Finally, from each $P_i$, we only lose at most $2\eps L$ vertices, hence the length of the path will be at least $$(L-2\eps L)(1-2\eps)N  = (1-2\eps)^2|F|\ge \left(2-O(\eps)\right)\frac{n}{d}\log d,$$ completing the proof.
	\endproof
	
	\begin{remark}\label{rem:cycle}
		The condition in Lemma~\ref{lem:DFS} actually also implies the existence of a cycle of length at least $N-4k+4$ (since there is an edge from the last $k$ vertices on the obtained path to the first $k$ vertices). Using this, in the above proof, we can connect the paths of $F$ into an induced cycle of length $(2-O(\eps))\frac{n}{d}\log d$.
	\end{remark}

	\section{Concluding remarks}
	
	\begin{itemize}
		
		\item Our proof is not constructive, since the first part of the argument uses the second moment method.
		The previously best bound $\sim\frac{n}{d}\log d$ due to \L{}uczak~\cite{luczak:93} and Suen~\cite{suen:92} was obtained via certain natural algorithms.
		It seems that this could be a barrier for such approaches. A (rather unsophisticated) heuristic giving evidence is that when we have grown an induced tree of this size, and assume the edges outside are still random, then the expected number of vertices which could be attached to a given vertex of the tree is less than one for $|V(T)|\ge(1+\eps)\frac{n\log d}{d}$.
		Moreover, such an ``algorithmic gap'' has been discovered for many other natural problems.
		In particular, despite decades of research, no polynomial-time algorithm is known which finds an independent set of size $(1+\eps)\frac{n}{d}\log d$ for any fixed $\eps>0$, and evidence has emerged that in fact such an algorithm might not exist (see e.g.~\cite{COE:15,GS:17,RV:17}).

		\item 
		In~\cite{CDKS:ta} it is conjectured that one should not only be able to find an induced path of size $\sim 2\frac{n}{d}\log d$, but an induced copy of any given bounded degree tree of this order. For dense graphs, when $d=\omega(n^{1/2}\log n)$, this follows from the second moment method (see~\cite{draganic:20}). In fact, Lemma~\ref{lem:2ndmoment} shows that the maximum degree can even be a small polynomial. On the other hand, the sparse case seems to be more difficult, mainly because the vanilla second moment method does not work.
		However, Dani and Moore~\cite{DM:11} demonstrated that one can actually make the second moment method work, at least for independent sets, by considering a \defn{weighted} version. This even gives a more precise result than the classical one due to Frieze~\cite{frieze:90}. It would be interesting to find out whether this method can be adapted to induced trees.

	\end{itemize}
	
	\section*{Acknowledgement}
	
	The first two authors are grateful to Benny Sudakov for very useful discussions.

	\providecommand{\bysame}{\leavevmode\hbox to3em{\hrulefill}\thinspace}
	\providecommand{\MR}{\relax\ifhmode\unskip\space\fi MR }
	\providecommand{\MRhref}[2]{%
		\href{http://www.ams.org/mathscinet-getitem?mr=#1}{#2}
	}
	\providecommand{\href}[2]{#2}

\end{document}